\documentclass{amsart}

\usepackage[utf8]{inputenc}
\usepackage{xcolor}
\usepackage{amsthm,amssymb}
\usepackage{fullpage}
\usepackage{graphicx}
\usepackage{nicefrac}
\usepackage{tikz}
\usepackage{subfig}
\usepackage{gensymb}
\usepackage[normalem]{ulem}

\theoremstyle{plain}
\newtheorem{theorem}{Theorem}
\newtheorem{lemma}{Lemma}

\newtheorem*{moore}{Irregular Moore Bound}
\newtheorem{proposition}{Proposition}
\newtheorem{definition}{Definition}

\newtheorem*{ear}{Ear Decomposition}

\theoremstyle{definition}
\newtheorem{example}{Example}

\newcommand{\set}[1]{\left\{ #1 \right\}}
\newcommand{\size}[1]{\left| #1 \right|}
\newcommand{\ceil}[1]{\left\lceil #1 \right\rceil}
\newcommand{\floor}[1]{\left\lfloor #1 \right\rfloor}
\newcommand{\paren}[1]{\left( #1 \right)}
\newcommand{\ceg}{universal cyclic edge-connectivity}

\tikzset{vtx/.style={circle, inner sep=0pt, minimum size=0.22cm,fill=black}}

\DeclareMathOperator{\cegK}{\kappa^\prime_{\circ}}

\title{Spectral Threshold for Extremal  Cyclic Edge-Connectivity}
\author{Sinan G.\ Aksoy}
\address{Pacific Northwest National Laboratory \\ Seattle, WA, 98109}
\email{sinan.aksoy@pnnl.gov}
\author{Mark Kempton}
\address{Mathematics Department \\ Brigham Young University \\ Provo, UT, 84602}
\email{mkempton@mathematics.byu.edu}
\author{Stephen J.\ Young}
\address{Pacific Northwest National Laboratory \\ Richland, WA, 99354}
\email{stephen.young@pnnl.gov}
\date{\today}
\thanks{This work was supported by the High Performance Data Analytics (HPDA) program at Pacific Northwest National Laboratory.  Pacific Northwest National Laboratory is operated by Battelle Memorial Institute under Contract DE-ACO6-76RL01830. \textit{PNNL Information Release:} PNNL-SA-151831}
\begin{document}
\begin{abstract}
%\sinan{Needs rewrite}
The \ceg\ of a graph $G$ is the least $k$ such that there exists a set of $k$ edges whose removal disconnects $G$ into components where every component contains a cycle.  We show that for graphs of minimum degree at least 3 and girth $g$ at least 4, the \ceg\ is bounded above by $(\Delta-2)g$ where $\Delta$ is the maximum degree.  We then prove that if the second eigenvalue of the adjacency matrix of a $d$-regular graph of girth $g\geq4$ is sufficiently small, then the \ceg\ is $(d-2)g$, providing a spectral condition for when this upper bound on \ceg\ is tight.
\end{abstract}

\maketitle

\section{Introduction}
The traditional notion of graph edge-connectivity is the smallest $k$ such that there exists a set of edges $S \subseteq E(G)$ with $|S|=k$, where $G \setminus S$ is disconnected. Note that traditional edge-connectivity does not stipulate any conditions on properties of the components of $G \setminus S$. The notion of {\it conditional edge-connectivity}, introduced by Harary \cite{Harary1983}, extends the traditional edge-connectivity by stipulating that components of $G \setminus S$ satisfy some given property. More precisely:

\begin{definition}[Harary 1983 \cite{Harary1983}]
Let $P$ be any property of a graph $G=(V,E)$, and let $S \subset E(G)$. The {\it universal $P$-connectivity} is the minimum $|S|$ such that $G \setminus S$ is disconnected, and every component of $G \setminus S$ has property $P$.
\end{definition}

We note there are several formulations of conditional connectivity; for instance, the qualifier {\it universal} reflects that {\it every} component of $G \setminus S$ has $P$, whereas {\it existential} conditional connectivity relaxes this condition to {\it some} component satisfying $P$.
{Harary's introduction of conditional connectivity focused on surveying a wide swath of possible properties, including planarity, degree and diameter restrictions, and the property of being cyclic, Hamiltonian, or Eulerian. 
Aruguing that ``the most fruitful topics are those suggested by applications", Harary aimed to provide a framework for devising connectivity concepts that are meaningful in applications. In particular, Harary references the pertinence of conditional connectivity to the analysis of
computer network reliability \cite{boesch1976large}, and VLSI and separator problems \cite{Lipton1979}, among others.
For instance, he notes that guarantees on connected component sizes reflect the resilience of computer networks to disruptions.
}

% As discussed in \cite{Harary1983}, Harary introduced conditional connectivity with explicit hopes of providing a framework for devising connectivity concepts that are meaningful in applications. Indeed, Harary notes that conditional connectivity for various properties has naturally arisen in areas such as computer network reliability \cite{boesch1976large}, VLSI and separator problems \cite{Lipton1979}, among others. 

In this work, we consider the universal $P$-edge-connectivity, where $P$ is the property of containing a cycle. We call this the \ceg\ of $G${{, which we denote by $\cegK(G)$}}. Prior to Harary's work, Bollobas alludes to \ceg\ \cite[p.~113]{bollobas2004extremal}, and Harary proved the \ceg\ of the balanced complete bipartite graph $K_{n,n}$ is $n^2-2n$ for $n$ even. The cycle condition is also natural for a number of applications, such as network reliability, as the existence of a cycle is necessary to guarantee multiple paths between pairs of vertices. { Cyclic edge-connectivity has also been utilized as a condition to solve other problems, such as in integer flow conjectures \cite{fan1992integer}, and early attempts \cite{tait1880remarks} at proving the four color theorem. Lastly, cyclic-edge connectivity has also garnered interest due to its close relationships with other types of connectivity \cite{latifi1994conditional, peroche1983several} and so-called $n$-extendable graphs \cite{plummer1980n}. }

{ Before proceeding, we note universal cyclic edge-connectivity is equal to several other notions of connectivity. Instead of requiring {\it every} component of $G\setminus S$ have a cycle, a number of researchers (cf. \cite{liang2020distributed, lou2005efficient, lou2008characterization}) define cyclic edge-connectivity has the smallest edge cut set $S$ such that {\it at least 2} components possess a cycle. Defined in this way, while a cyclic edge cut need not be a universal cyclic edge cut, it is nonetheless straightforward\footnote{Suppose that $S$ is a minimal cyclic edge cut that is not universal.  Then $G\setminus S$ has a tree component $T$ as well as two components $C_1$ and $C_2$ that contain a cycle.  Now there is some edge $e \in S$ which is incident to $T$.  Consider the edge cut $S' = S\setminus e$.  Since $e$ is not incident between $C_1$ and $C_2$, $G \setminus S'$ still has two distinct components, $C_1'$ and $C_2'$ which contain $C_1$ and $C_2$ as subgraphs, respectively.  But then $S'$ is a cyclic edge cut, a contradiction.} to show this notion of cyclic-edge connectivity is equal to universal cyclic edge-connectivity. Despite this equivalence, we still utilize the universal formulation of cyclic edge-connectivity because restricting the permissible cuts affords advantages in our proof techniques. Lastly, we note other notions of edge-connectivity are equivalent to cyclic edge-connectivity for certain parameter settings, or for certain families of graphs. For instance, Latifi et.~al.~\cite{latifi1994conditional} propose a new measure of conditional connectivity for large multiprocessor systems, requiring every vertex have degree $k$ in $G \setminus S$. As observed in \cite{wang2009cyclic}, when $k=2$ and the minimum degree is at least 3, this quantity is equal to cyclic edge-connectivity. For more on the relationship between cyclic edge-connectivity and other connectivity measures, see \cite{peroche1983several}. }

% \sinan{An alternative to universal is two component...cue literature review on cyclic edge connectivity. This gives different cuts, but same value. Our method affords advantages in proof by reducing space of cuts.}

We take a spectral approach for studying universal cyclic edge-connectivity. We prove that for a $d$-regular graph with girth $g$, with $d\geq5$, a bound on the second eigenvalue of the adjacency matrix is sufficient to guarantee that $\cegK(G)=(d-2)g$, which is the largest possible \ceg\ (see Theorems \ref{T:upper} and \ref{thm:main} below). Furthermore, we construct a family of $d$-regular graphs that show our spectral condition is necessary. Prior work has established a number of spectral bounds for traditional edge and vertex connectivity; see { \cite{Abiad2018,cioabua2010eigenvalues,fiedler1973algebraic}} and the references contained therein. { Furthermore, researchers have also investigated connectivity and conditional connectivity with restrictions on component sizes for many families of graphs \cite{cioabua2012conjecture,cioabua2014disconnecting,fabrega1994extraconnectivity,liu2010existence,zhang2019cyclic}}. 
However, {spectral} bounds on conditional connectivity appear far more rare.
{ One exception, however, is recent work by Zhang \cite{zhang2019cyclic}, who makes use of spectral tools to determine the cyclic edge-connectivity of strongly regular graphs. 
Despite the pervasive use of spectral methods as a proof technique, Zhang's approach does not immediately lead to an eigenvalue condition for cyclic edge-connectivity.
Focusing on the more general case of $d$-regular graphs, our bound establishes a spectral threshold for extremal cyclic edge-connectivity, depending on girth and degree. We note all strongly-regular graphs which are at least 5-regular satisfy the condition of our theorem.  It is interesting to note, that all 4-regular strongly regular graphs also satisfy the stated spectral threshold even though not satisfying our degree condition.
}

%Most strongly regular graphs satisfy the conditions of our main theorem, and thus our results imply Zhang's for a wide variety of graphs. 

%We note that for most strongly regular graphs, the result of Zhang is an easy consequence of our result. 

%However, {spectral} bounds on conditional connectivity appear far more rare and, to our knowledge, no such results exist for \ceg. 

\section{Main tools and notation}
For a graph $G=(V,E)${, let $\cegK(G)$ denote the cyclic edge-connectivity of $G$. For} vertex subsets $X,Y\subseteq V$, let $E(X,Y)$ denote the set of edges between $X$ and $Y$, and let $e(X,Y):=|E(X,Y)|$. Further, let $G[X]$ denote the subgraph induced by $X$. {One of our primary tools will be the following well-known lemma from \cite{Mohar:LapEig}.}
%A primary tool we will use is the discrepancy inequality, also sometimes called the expander mixing lemma. 

%\begin{EML}[Alon and Chung~\cite{Alon1988}]
%Let $G=(V,E)$ be a $d$-regular, $n$-vertex graph with second largest adjacency eigenvalue $\lambda$. For all $X,Y \subseteq V$, 
%\[
%\left|e(X,Y) -\frac{d|X||Y|}{n}\right| \leq \lambda \sqrt{|X||Y|\left(1-\frac{|X|}{n}\right)\left(1-\frac{|Y|}{n}\right)}.
%\]
%\end{EML}

{
\begin{lemma}[Lemma 3.1 of \cite{Mohar:LapEig}]\label{lem:edgebound}
Let $G$ be a graph of order $n$, $X\subseteq V$, and $\lambda_2$ the second largest eigenvalue of the adjacency matrix.  Then \[e(X,\bar X) \geq (d-\lambda_2)\frac{|X|(n-|X|)}{n}.\]
\end{lemma}
}

The other key ingredient of the proof is the following theorem of Alon, Hoory, and Linial which provides a lower bound for the number of vertices in a graph with a given average degree and girth.  This result may be thought of as an irregular generalization of the result of Moore~(see \cite[p.~180]{Biggs:AGT}), which lower bounds the number of vertices in a $d$-regular graph of a given diameter.
\begin{moore}[Alon, Hoory, Linial \cite{Alon2002}]
%Alon, et al. theorem on size of large girth graphs with average degree d.  
The number of vertices $n$ in a graph of girth $g$ and average degree at least $d\geq 2$ satisfies $n \geq n_0(d, g)$ where
%\begin{align*}%\label{eq:n0}
\[ n_0(d,g)=\begin{cases} \displaystyle 1+d\sum_{i=0}^{r-1}(d-1)^i & \mbox{ if } g=2r+1 \\
\displaystyle 2\sum_{i=0}^{r-1}(d-1)^i & \mbox{ if } g=2r\end{cases}.
\]%\end{align*}
\end{moore}

To demonstrate the applicability of these tools, we first obtain the following na\"{i}ve spectral bound on conditional edge-connectivity in terms of girth when conditioned on the size of the minimal component. 

\begin{proposition}\label{P:trivial}
Let $\gamma$ be the size (number of edges) of the smallest edge-cut of a $d$-regular graph { with second largest adjacency eigenvalue $\lambda_2$} which results in components of size at least $k$, and let the girth be $g$. Then
\[
\gamma \ge k(d-{\lambda_2})\left(1-\frac{k}{n_0(d,g)}\right),
\]
{where $n_0(d,g)$ is as in the Irregular Moore Bound.}
\end{proposition}
\begin{proof}
We first note that by the minimality of the cut, there is some set $X$ such that $\gamma = e(X,\overline{X})$.  Thus, 
by { Lemma \ref{lem:edgebound}}, %the Discrepancy Inequality,  
we have that
\[
\gamma = e(X,\bar X) \geq %\frac{d|X|(n-|X|)}{n}-\lambda\sqrt{|X|(n-|X|)\left(1-\frac{|X|}{n}\right)\left(1-\frac{n-|X|}{n}\right)}=
(d-{\lambda_2})\frac{|X|(n-|X|)}{n}.
\]
Applying the aforementioned Irregular Moore Bound of Alon, Hoory, and Linial~\cite{Alon2002} immediately yields the desired result.
\end{proof}

The form of this lower bound is quite close to the trivial upper bound on $\gamma$, $k(d-2)+2$, stemming from the case where $G[X]$ is a $k$-vertex tree. In some sense, this extremal example makes \ceg\ the natural strengthening of the size-limited connectivity.

The final tool we will utilize to study the \ceg\ of $d$-regular graphs is the {\it ear decomposition} of a graph, which as stated in \cite{Bondy2010}, may defined as follows. 
\begin{ear}
For a subgraph $F$ of $G$, an {\it ear} of $F$ in $G$ is a nontrivial path in $G$ whose ends lie in $F$ but whose internal vertices do not. An {\it ear decomposition} of a 2-edge-connected graph $G$ is a nested sequence $(G_0,G_1,\dots,G_k)$ of subgraphs of $G$ such that
\begin{itemize}
    \item[$(i)$]  $G_0$ is a cycle,
    \item[$(ii)$]  $G_{i+1}=G_i \cup P_i$, where $P_i$ is an ear of $G_i$ in $G$, for $0\leq i\leq k$,
    \item[$(iii)$] $G_k=G$.
\end{itemize}
\end{ear}

\section{Cyclic Edge-Connectivity}\label{S:cyclic}

Before proceeding with our main result, we first address the existence of and upper bounds on \color{black} cyclic-edge connectivity. Indeed, for acyclic or unicyclic graphs, it is clear the cyclic edge connectivity does not exist. The work of Lov\'{a}sz~\cite{lovasz65} and Dirac~\cite{dirac63} provide a complete characterization of the class of graphs with no pairs of vertex disjoint cycles.  Roughly speaking, these are graphs obtained from $K_5$, a wheel, and $K_{3,t}$ plus any subset of edges connecting vertices in the three element class, and a forest plus a dominating vertex by the duplication and subdivision of edges and the addition of trees.  However, as no assurances are provided on the relative sizes of these cycles, these works do not yield an upper bound on the \ceg\ even in the $d$-regular case.  Lou and Holton \cite[Lemma 1]{lou1993lower} provide a straight-forward argument that for $d$-regular girth $g$ graphs, the cyclic edge connectivity is bounded above by $(d-2)g$.  However, their approach does not appear to generalize to the irregular graphs.  To address this, we provide in Lemma \ref{L:upper} explicit conditions for the existence of a girth-length cycle which induces a universal cyclic edge cut.

\color{black}
\begin{lemma}\label{L:upper}
Let $G$ be a graph with minimum degree $d \geq 3$ and girth $g \geq 4$.  If $G$ is not $K_{3,t}$, then there exists a cycle $C$ of length $g$ such that every component of $G - E(C,\overline{C})$ contains a cycle.
\end{lemma}

As an immediate corollary we have the following.

\begin{theorem}\label{T:upper}
Let $G$ be a graph with minimum degree at least 3 and girth $\geq 4$. If $G$ is not $K_{3,t}$, then $\cegK(G) \leq \paren{\Delta-2}g$ where $\Delta$ is the maximum degree of $G$.
\end{theorem}
\color{black}

\begin{proof}[Proof of Lemma \ref{L:upper}]
We first consider the case that $g \geq 5$ and let $C$ be a cycle of length $g$.  
As $C$ is an induced cycle it is clear that $G[C]$ contains a cycle.  Now consider an arbitrary component $G[X]$ of $G - E(C,\overline{C})$ and let $x \in X$.  As $g \geq 5$, the vertex $x$ has at most one neighbor in $C$ as otherwise there exists a cycle of length at most $\floor{\frac{g}{2}} + 1 < g$.   But then the minimum degree of $G[X]$ is 2 and  hence it contains a cycle. 

Now suppose that $G$ is a graph with girth $4$ and is not $K_{3,t}$.  Let $x,y$ be vertices of $G$ such that the set of common neighbors, $Z = \set{z_1,\ldots, z_k}$ has size at least 2.  The existence of such a pair of points is guaranteed as the girth is 4 and there is an induced 4-cycle in $G$.  It is worth mentioning that, since the girth of $G$ is 4, none of the vertices adjacent to a vertex in $Z$ are adjacent to either $x$ or $y$.  We now consider the component structure of $G - \set{x,y}$.  Let $Z_1, \ldots, Z_{k_z}$ be the vertex sets for components that contain an element of $Z$ and let $X_1,\ldots, X_{k_x}$, (respectively $Y_1, \ldots, Y_{k_y}$) be the vertex sets for the components such that $E(X_i,\set{x}) \neq \varnothing$ and $X_i \cap Z = \varnothing$ (respectively, $E(Y_i,\set{y}) \neq \varnothing$ and $Y_i \cap Z = \varnothing$).  We note that the collection of vertex sets $\set{X_1,\ldots X_{k_x}}$ and $\set{Y_1, \ldots, Y_{k_y}}$ are not necessarily distinct, however, this potential duplicate naming will not affect our subsequent analysis. We will show the desired cycle is in $\set{x,y} \cup \bigcup_i Z_i$.  We note the components induced by any $X_i$ or $Y_j$ do not provide an obstruction. This is easy to see as by definition any vertex in $X_i$ or $Y_j$ is incident to precisely one of $\set{x,y}$ as otherwise it would belong to $Z$.  Thus $G[X_i]$ and $G[Y_j]$ have minimum degree at least 2, and thus, contain a cycle. As a consequence, we may restrict our attention to the components $G[Z_1], \ldots, G[Z_{k_z}]$ without loss of generality. 

%Ideally, the choice of any $z,z' \in Z$ would yield a cycle $C = \set{x,z,y,z'}$ such that the all the components of $G - C$ contain cycles, and thus providing the desired cycle.  While we will not be able to show this ideal case, we will in fact be able to restrict ourselves to cycles lying in $\set{x,y} \cup \bigcup_i Z_i$.  Thus, we will first show that the components induced by any $X_i$ or $Y_j$ do not provide an obstruction.  However, this is easy to see as by definition any vertex in $X_i$ or $Y_j$ is incident to precisely one of $\set{x,y}$ as otherwise it would belong to $Z$.  Thus $G[X_i]$ and $G[Y_j]$ have minimum degree at least 2, and thus, contain a cycle.

 We first consider the case where one of the components, say $T = G[Z_1]$, is a tree.  As the minimum degree of $G$ is 3 and $Z$ contains every vertex adjacent to both $x$ and $y$, the leaves of $T$ are given by $Z_1 \cap Z.$  Let $z,z'$ be two vertices of maximum distance in $T$ and let the unique path between them be given by $z = t_0, t_1, \ldots, t_{\ell}, t_{\ell+1} = z'$.  We first consider the case where $\ell \geq 2$ and so $t_1$ and $t_{\ell}$ are distinct vertices.  We note that $E(\set{x,y},\set{t_1,t_{\ell}}) = \varnothing$ as both $t_1$ and $t_{\ell}$ are incident to elements of $Z$.  Thus $t_1$ and $t_{\ell}$ have degree at least 3 in $T$ and thus there are vertices $t_1'$ and $t_{\ell}'$ that are incident to $t_1$ and $t_{\ell}$, respectively.  By the maximality of the distance between $z$ and $z'$ in $T$ and the uniqueness of the shortest path in trees, $t_1'$ and $t_{\ell}'$ are also leaves and hence in $Z$.  But then consider the components of $G -C$, where $C$ is the cycle $\set{y,z,t_1,t_1'}$.  Note that $\set{x,z',t_{\ell},t_{\ell}'}$ is a cycle and disjoint from $\set{y,z,t_1,t_1'}$ and thus is present in $G - C$.  Furthermore, the components of $T - \set{z,t_1,t_1'}$ as well as the components $G[Z_2], \ldots, G[Z_{k_z}]$ are all incident to $x$ and thus form a single component which contains the cycle $\set{x,z',t_{\ell},t_{\ell}'}$.  Hence, $C$ is the desired cycle.

Thus we may now assume that any tree component among $G[Z_1], \ldots, G[Z_{k_z}]$ has diameter 2 and a unique vertex not in $Z$.  Suppose $G[Z_1]$ is such a component and let $v$ be the unique element of $G[Z_1]$ not in $Z$.  Since $v$ is adjacent to an element of $Z$, we have that $v$ is not adjacent to $\set{x,y}$, has degree at least 3, and $\set{x,y,v} \cup (Z_1 \cap Z)$ induces copy of $K_{3,t}$ in $G$.  As $G$ is not equal to $K_{3,t}$, this implies that one of $Z_2$, $X_1$, or $Y_1$ exists and is not empty.  Suppose first that $X_1$ exists and let $z,z' \in Z_1 \cap Z.$  Consider the components of $G - \set{y,z,v,z'}$.  As every element of $Z - \set{z,z'}$ is incident to $x$, $G[X_1]$ contains a cycle, and there is an edge between $X_1$ and $x$, we have that every component formed contains a cycle. A similar argument holds if $Y_1$ exists by exchanging the roles of $x$ and $y.$  Finally, assume that no components of the type $X_i$ or $Y_j$ exist, but $Z_2$ exists. Let $z,z' \in Z \cap Z_1$ and consider the components of $G - \set{y,z,v,z'}.$   As all the components $Z_3,\ldots, Z_{k_z}$ as well as the vertices of $Z_1  - \set{z,z'}$ are connected to $x$, in order to show that the all the components of $G -\set{y,z,v,z'}$ have a cycle it suffices to show that $G[Z_2 \cup \set{x}]$ contains a cycle.  To that end, if $\size{Z_2 \cap Z} = 1$, then every vertex in $Z_2 - Z$ is adjacent to at most one vertex in $\set{x,y} \cup Z$ and hence $G[Z_2 - Z]$ has minimum degree 2 and a cycle. Otherwise, $\size{Z_2 \cap Z} \geq 2$ and $G[ Z_2 \cup \set{x}]$ contains a cycle by taking a path between distinct elements of $Z$ in $G[Z_2]$ joined by the vertex $x$.  

At this point, we may assume without loss of generality that $G[Z_i]$ is not a tree for all $i$. Suppose that the induced graph $G[\bigcup_i Z_i]$ has at least two vertices of degree 1, $z,z'$.  Then as $z$ and $z'$ are on no cycles in $G[\bigcup_i Z_i]$, the components of $G[\bigcup_i Z_i - \set{z,z'}]$ all contain cycles.  This gives that $\set{x,z,y,z'}$ is the desired cycle $C$.  Thus we may assume that there is at most 1 vertex of degree 1 in $G[\bigcup_i Z_i]$, which we call $z$, and let $z' \in Z - \set{z}$.  As the induced subgraph $G[\bigcup_i Z_i - \set{z,z'}]$ has at most one vertex of degree 1 (potentially the unique common neighbor of $z$ and $z'$), all of its components contain a cycle and again $\set{x,z,y,z'}$ is the desired cycle.  
We may now assume that for $i$, the induced subgraph $G[Z_i]$ is not a tree and every element of $Z$ has degree at least 2 in the relevant component.  Suppose that $k_{z} \geq 2$ and let $z \in Z_1 \cap Z$, $z' \in Z_2 \cap Z$.  Every vertex in $Z_1 - Z$ is adjacent to at most one of $\set{x,y,z,z'}$ and so has degree at least 2 in $G[Z_1 - \set{z,z'}]$, while the elements of $Z\cap Z_1 - \set{z,z'}$ are incident to neither of $z$ or $z'$, and thus also have minimum degree 2.  This implies that $G[Z_i - \set{z,z'}]$ has minimum degree 2 for all $i$ and thus contains a cycle, and hence $\set{x,z,y,z'}$ is the desired cycle.

Finally, we may now assume that there is a single component $Z_1$ of $G - \set{x,y}$ that contains all elements of $Z$ and further, that component is not a tree and every element of $Z \cap Z_1$ has degree at least two in the component. Now fix two elements $z,z' \in Z$ and and let $F$ be the forest of tree components of $G[Z_1 - \set{z,z'}]$.  Clearly if $F$ is empty, then $\set{x,z,y,z'}$ is the desired cycle.  Thus we may assume that $F$ is non-empty and let $L$ be the leaves of $F$.  As $Z$ is an independent set and every vertex of $z$ has degree at least 2 in $G[Z_1]$, we note that $L \cap Z = \varnothing$.  Further, as the minimum degree is at least 3, any $\ell \in L$ is adjacent to at least two of $\set{x,y,z,z'}$.  As $\ell \not\in Z$, it can not be adjacent to both $x$ and $y$.  Additionally, $\ell$ can not be adjacent to one of  $\set{x,y}$ and one of $\set{z,z'}$ as this forms a triangle.  Thus every element of $L$ is adjacent to both $z$ and $z$'.  Now suppose that there is some tree $T \in F$ such that $E(T,\set{x,y}) = \varnothing$ and let $\ell,\ell'$ be leaves of that tree.  But then, $z$ and $z'$ are antipodal points in a 4-cycle such that their common neighbors are in distinct components of $G - \set{z,z'}$.  Specifically, $\ell,\ell'$ are in a different component than $\set{x,y}$ and thus by previous arguments the desired cycle exists.  Thus we may assume that every component of $F$ is adjacent to either $x$ or $y$.  But then, if $\size{Z} \geq 4$, we have that for any two leaves $\ell,\ell' \in L$, the cycle $\set{z,\ell,z',\ell'}$ is the desired cycle.  Specifically, if $\bar{z},\bar{z}' \in Z - \set{z,z'}$ then the tree components of $G - \set{x,y,z,z',\ell,\ell'}$ are all connected to the cycle $\set{x,\bar{z},y,\bar{z}'}$ via either $x$ or $y$.  To complete the proof we note that the common neighbors of $z$ and $z'$ include $\set{x,y}$ and $L$, and thus have at least 4 elements.  Thus, by repeating the arguments above with $\set{z,z'}$ in the role of $\set{x,y}$, we may assume that $\size{Z} \geq 4$ as needed.
\end{proof}

One might hope Lemma \ref{L:upper} could be extended to girth 3 graphs with a similarly small set of exceptions as the girth 4 case. However, it is relatively easy to identify infinite families of counterexamples  from the work of Lov\'asz~\cite{lovasz65} and Dirac~\cite{dirac63}\color{black}. For example, the wheel graph on (see Figure \ref{F:3girth}) on any number of vertices forms a counterexample as every 3-cycle in the graph involves the central hub as well as an edge from the from the outer cycle.  Thus the removal of a 3-cycle destroys every cycle in the graph.  Another infinite family of counterexamples can be formed by taking $K_{3,t}$ and adding a non-empty set of edges to the partition of size 3.  In this case, every 3-cycle uses 2 of the vertices of partition of size 3 and thus because the graph is bipartite there are not enough vertices remaining on that side to form a cycle (see Figure \ref{F:3girth}).  

 In fact, by adding relatively few vertices it is possible to transform an arbitrary triangle-free graph into a counterexample to Lemma \ref{L:upper} where the \ceg\ exists.  Specifically, let $G$ be an arbitrary triangle-free graph and let $T$ be a tree with at least 3 leaves and minimum non-leaf degree 3.  Adding two adjacent vertices $c,c'$ which are connected to all the leaves of $T$ and an arbitrary independent set $S$ in $G$, results in a graph $G'$ in which every triangle uses the edge $\set{c,c'}$ and a vertex in $S$ or a leaf of $T$.  Thus, deleting the edges incident to any triangle in $G'$ results a cycle-free component, namely $T$ or $T$ with a single leaf removed.\color{black}
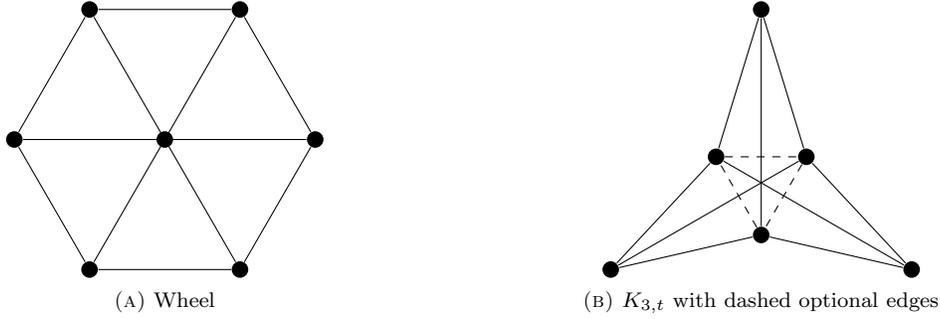
\begin{figure}
    \centering
    \hfill
    \subfloat[][Wheel]{
    \qquad\qquad
    \begin{tikzpicture}
    \node[vtx] (a) at (1,0) {};
    \node[vtx] (b) at (3,0) {};
    \node[vtx] (c) at (4,1.73) {};
    \node[vtx] (d) at (3,3.46) {};
    \node[vtx] (e) at (1,3.46) {};
    \node[vtx] (f) at (0,1.73) {};
    \draw (a) -- (b) -- (c) -- (d) -- (e) -- (f) -- (a) -- (d);
    \draw (b) -- (e);
    \draw (c) -- (f);
    \node[vtx] (H) at (2,1.73) {};
    \end{tikzpicture}
    \qquad\qquad
    }
    \hfill
    \subfloat[][$K_{3,t}$ with dashed optional edges]{
    \qquad\qquad
    \begin{tikzpicture}
    \node[vtx] (A) at (0,0) {}; 
    \node[vtx] (B) at (4,0) {};
    \node[vtx] (C) at (2,3.46) {};
    \node[vtx] (a) at (1.4,1.5) {};
    \node[vtx] (b) at (2.6,1.5) {};
    \node[vtx] (c) at (2,0.46) {};
    \draw (A) -- (c) -- (B) -- (b) -- (C) -- (a) -- (A);
    \draw[dashed] (a) -- (b) -- (c) -- (a);
    \draw (A) -- (b);
    \draw (B) -- (a);
    \draw (C) -- (c);
    \end{tikzpicture}
    \qquad\qquad
    }
    \hfill\phantom{}
    \caption{Examples from two infinite families of girth 3 graphs where the \ceg\ does not exist.  \label{F:3girth}}
\end{figure}
However, with the added restriction that the graph is $d$-regular, the only counterexamples the authors are aware of are $K_4$ and $K_5$.  Thus it is possible there is a finite set of exceptions for a $d$-regular, girth 3, version of Lemma \ref{L:upper}.

\begin{theorem} \label{thm:main}
{ Let $G$ be a $d$-regular graph with $d \geq 5$ and second largest adjacency eigenvalue $\lambda_2$.  If $G$ has girth $g\geq 4$ and 
\[d-\lambda_2 \geq \frac{2(d-2)g}{n_0\paren{d - \frac{2}{r-1},g}},\]
where $r = \floor{\nicefrac{g}{2}}$ and $n_0$ is as in the Irregular Moore Bound, then $\cegK(G) = (d-2)g$.
If $G$ has girth $g=3$, the cyclic-edge connectivity exists, and 
\[
d-\lambda_2  \geq  6 - \frac{12}{d},
\] then  $\cegK(G) = 3d-6$.}
\end{theorem}

% \begin{theorem} \label{thm:main}
% Let $G$ be a $d$-regular graph with $d \geq 5$ and second largest adjacency eigenvalue {{$\lambda_2$}}.  If the \ceg\ exists, the girth is at least 4, and 
% \[d-\lambda_2 \geq \frac{2(d-2)g}{n_0\paren{d - \frac{2}{r-1},g}}\]
% %\[\frac{2(d-2)g}{d-\lambda} \leq n_0\paren{d - \frac{2}{r-1},g}\] 
% where $r = \floor{\nicefrac{g}{2}}$ \textcolor{blue}{and $n_0$ is defined in Eq. \ref{eq:n0} of the Irregular Moore Bound}, then 
% \[
% {
% \cegK(G) \geq (d-2)g.}
% \]
% If the girth is 3 and {$d-\lambda_2  \geq  6 - \nicefrac{12}{d}$}, then $\cegK(G) \geq 3d-6$.
% \end{theorem}
\begin{proof}
%\sjy{This is a sketch at this point}
Let $S$ be a minimal set of edges such that $G-S$ is disconnected with all components containing a cycle and that $\size{S} < (d-2)g$.  By minimality, we may assume that there is some set of vertices $X$ such that $S = E(X,\overline{X})$, $\size{X} \leq \size{\overline{X}}$, and $H = G[X]$ contains a cycle.  Further, we may assume that $X$ is the minimal cardinality set yielding an edge cut of size $\size{S}$.  Now, by {Lemma \ref{lem:edgebound}}, we have that 
\[ e(X,\overline{X}) \geq (d-{{\lambda_2}}) \frac{\size{X} \size{\overline{X}}}{n} \geq (d-{{\lambda_2}}) \frac{\size{X}}{2}.\]  Thus, to prove the desired result for $g\geq 4$, { by Lemma \ref{L:upper} (or Lemma 1 in \cite{lou1993lower})} it suffices to show that
\[ \size{X} \geq n_0\paren{d-\frac{2}{r-1},g} \geq \frac{2(d-2)g}{d-{{\lambda_2}}}.\]  

We first observe that if $H$ is a simple cycle, then $\size{X} \geq g$ and $e(X,\overline{X}) \geq (d-2)g$, as desired.  Thus assume without loss of generality that $H$ is not a cycle.  Further note that, if $x \in X$ is such that there is some cycle $C \in G[X - \set{x}]$, then the degree of $x$ in $H$ is at least $\ceil{\frac{d}{2}}$ %\kempton{should this be set off as its own lemma or claim so that it is easier to refer to later on?} 
as otherwise $G[X - \set{x}]$ contains the cycle $C$, $G[\overline{X}\cup\set{x}]$ contains the same cycle as $G[\overline{X}]$, and $e(X-\set{x},\overline{X}\cup \set{x}) \leq e(X,\overline{X})$. As a consequence, the minimum degree in $H$ is at least 2.  

At this point it is possible to observe that $H$ is 2-edge-connected.  Specifically, suppose that there exists two disjoint sets $X_1$ and $X_2$ such that $X_1 \cup X_2 = X$ and $e(X_1,X_2) \leq 1$.  As the minimum degree in $H$ is at least 2, both $G[X_1]$ and $G[X_2]$ contain some cycle.   By minimality of the edge cut and that $X_1,X_2 \subset X$, we have that $e(X_1,\overline{X_1}),e(X_2,\overline{X_2}) \geq e(X,\overline{X})+1$, thus 
\[ e(X,\overline{X})  = e(X_1,\overline{X_1}) + e(X_2,\overline{X_2}) - 2 \geq 2e(X,\overline{X}),\] a contradiction. 

%From what we have shown so far, it follows that the average degree of $H$ is strictly larger than 2.  However, we will be able to say something stronger.  That is, there is a fixed $\epsilon>0$ independent of the size of the graph such that the average degree of $H$ is at least $2+\epsilon.$  We prove this below. \kempton{Check the above paragraph if that is how we want to express what we had talked about.}

Since $H$ is 2-edge-connected, there exists an ear decomposition for $H$~\cite{robbins1939theorem}. 
Specifically, there exists a cycle $C \in H$ as well as paths $P_1, \ldots, P_k$ such that the internal vertices of $P_i$ are disjoint from $C \cup \paren{\bigcup_{j=1}^{i-1} P_j}$ and $H = C \cup \paren{\bigcup_{j=1}^k P_j}$.  
Since we may assume that $H$ is not a cycle, we have that there is a non-zero number of paths in the ear decomposition.  
So we may consider the last path in the ear decomposition, $P_k$.  By construction of the ear decomposition, any internal vertex in $P_k$ will have degree 2 in $H$ but will be disjoint from the cycle $C$.  
But as such vertices have degree at least $\ceil{\frac{d}{2}}$ and  $d \geq 5$, no such vertex exists and $P_k$ is a single edge $e=\set{x,y}$.  
By the properties of the ear decomposition $H - e$  is a 2-edge-connected graph and hence there are at least two, edge-disjoint, paths between $x$ and $y$ in $H - e$, denote them by $Q$ and $Q'$.  
Without loss of generality we assume that the total length of $Q$ and $Q'$ is minimized.  
Now suppose there exists vertices $a$ and $b$ that are on both $Q$ and $Q'$, but in opposite orders.  
That is, $Q = x Q_1 a Q_2 b Q_3 y$ and $Q' = x Q_1' b Q_2' a Q_3'y$.  
We can then construct two new walks from $x$ to $y$, $x Q_1 a Q_3' y$ and $x Q_1' b Q_3 y$ which have total shorter length.   Thus, if $x= a_0, a_1, \ldots, a_{t-1}, a_{t} = y$ are the intersection points of $Q$ and $Q'$, they occur in the same order in both $Q$ and $Q'$.  
As a consequence, $Q \cup Q'$ can be thought of as a series of vertex incident cycles $C_0, C_1, \ldots, C_t$ such that $a_j,a_{j+1} \in C_j$ and $C_{i-1} \cap C_{i} = \set{a_i}$.  
Now for every vertex $v$ in $H$ (except potentially $x,y$ if $t = 1$ and $a_1$ if $t=2$), there is some cycle not containing $v$ and hence the degree of $v$ is at least $\ceil{\frac{d}{2}} \geq 3$.  
As the degree of $x$ and $y$ are at least $2$ and the degree of $a_1$ is at least $4$, this implies that average degree of $H$ is at least $2 + \epsilon$ for some $\epsilon > 0$. 

As $H$ has average degree $2+\epsilon$ and girth at least $g$, by the Irregular Moore Bound, we have that $\size{X} \geq n_0(2+\epsilon,g)$.  But then, since $G$ is $d$-regular and $e(X,\overline{X}) < (d-2)g$, we have the average degree is at least \[ d - \frac{(d-2)g}{n_0(2+\epsilon,g)}.\]  In particular, we have that $\epsilon$ must satisfy that \[(d-2-\epsilon) n_0(2+\epsilon,g) < (d-2)g.\] 

Consider first the case where $g = 2r \geq 4$, and note that
\begin{align*}
(d-2-\epsilon)n_0(2+\epsilon,2r) &= (d-2-\epsilon) \paren{ 2 \sum_{j=0}^{r-1} (1+\epsilon)^j} \\
&= 2(d-2-\epsilon) \frac{(1+\epsilon)^r - 1}{\epsilon} \\
&\geq 2(d-2-\epsilon)\paren{r + \binom{r}{2}\epsilon}.
\end{align*}
Thus we need to have \[ 2(d-2-\epsilon)\paren{r + \binom{r}{2}\epsilon} < 2r(d-2),\] which can be rearranged to 
\[ \binom{r}{2} \epsilon \paren{d - 2 - \frac{2}{r-1} - \epsilon} < 0.\]  As we already have that $\epsilon > 0$, this implies that $\epsilon > d - 2 - \frac{2}{r-1}$ and thus $\size{X} \geq n_0\paren{d - \frac{2}{r-1},2r}.$

Finally consider the case where $g = 2r + 1 \geq 5.$   In this case we have that 
\begin{align*}
    (d-2-\epsilon)n_0(2+\epsilon,2r+1) &= (d-2-\epsilon)\paren{1 + (2+\epsilon)\sum_{j=0}^{r-1}(1+\epsilon)^j }\\
    &\geq (d-2-\epsilon)\paren{1+(2+\epsilon)\sum_{j=0}^{r-1} (1+j\epsilon)} \\
    &= (d-2-\epsilon)\paren{1 + (2+\epsilon)\paren{r + \binom{r}{2}\epsilon}} \\
    &= (d-2-\epsilon)\paren{1 + 2r + \paren{2\binom{r}{2} + r}\epsilon + \binom{r}{2}\epsilon^2} \\
    &= (d-2-\epsilon)\paren{1+2r + r^2\epsilon + \binom{r}{2}\epsilon^2}
\end{align*}
Thus, we have that $\epsilon$ satisfies that
\[\epsilon\paren{ \paren{(d-2)r^2 - g} + \paren{(d-2)\binom{r}{2} - r^2}\epsilon - \binom{r}{2}\epsilon^2} < 0.\]
Letting \[ f(\epsilon) = -\binom{r}{2}\epsilon^2 + \paren{(d-2)\binom{r}{2} -r^2}\epsilon + (d-2)r^2 - g,\] it is easy to see that 
\begin{align*}
\lim_{\epsilon \rightarrow -\infty} f(\epsilon) &= -\infty, \\
f(0) &= (d-2)r^2 - 2r - 1 > 0, \\
f(d-2) &= -g,\quad\textrm{and} \\
\lim_{\epsilon \rightarrow \infty} f(\epsilon) &= -\infty\\
\end{align*}
Thus $f(\epsilon)$ has one root in $(-\infty,0)$ and one in $(0,d-2)$.  Let $\epsilon^*$ be the root of $f(\epsilon)$ in $(0,d-2)$, then we have that $\epsilon \geq \epsilon^*$ and $\size{X} \geq n_0(2 + \epsilon^*, 2r+1).$  Observing that,
\begin{align*}
    f\paren{d-2-\frac{2}{r-1}} &= -\binom{r}{2}\paren{d-2-\frac{2}{r-1}}^2 + \paren{(d-2)\binom{r}{2} - r^2}\paren{d-2-\frac{2}{r-1}} + (d-2)r^2 - g \\
    &= \binom{r}{2}\paren{-\paren{d-2-\frac{2}{r-1}}^2 + \paren{d-2-\frac{2}{r-1}}\paren{d-2-\frac{2r}{r-1}}+ \frac{2(d-2)r}{r-1} - \frac{g}{\binom{r}{2}}}\\
    &= \binom{r}{2}\paren{\paren{d-2-\frac{2}{r-1}}\frac{2-2r}{r-1}+ \frac{2(d-2)r}{r-1} - \frac{g}{\binom{r}{2}}}\\
    &= \binom{r}{2}\paren{\frac{2(d-2)}{r-1} + \frac{4r-4}{(r-1)^2}- \frac{4r +2}{r(r-1)}} \\
    &= \frac{2(d-2)(r-1)r + (4r-4)r - (4r+2)(r-1)}{2(r-1)} \\
    &= \frac{2(d-2)(r-1)r + 2 - 2r}{2(r-1)} \\
    &= (d-2)r - 1\\
    &> 0
\end{align*}
 
Thus $\epsilon^* > d - 2 - \frac{2}{r-1}$ and $\size{X} > n_0\paren{d - \frac{2}{r-1},2r+1}$.

Finally, the case for $g = 3$ proceeds similarly as above except that $n_0(2+\epsilon,3) = 3+\epsilon$ and thus $\epsilon^*$ can be determined to be exactly $d-5$, { yielding $\cegK(G) \geq 3d-6$.  The matching upper bound on $\cegK(G)$ for girth 3 follows immediately by applying Lemma 1 from \cite{lou1993lower}. }
\end{proof}

% \sout{Combining these results yields the following spectral condition for when the \ceg\ of a graph is given by a cut induced by a minimum length cycle.}

% \begin{corollary}
% \sout{{\color{red} If $G$ is a $d$-regular graph with $d \geq 5$, girth $g\geq 4$, and second largest eigenvalue of the adjacency matrix {$\lambda_2$}, satisfies
% \[ {d-{\lambda_2}} \geq \frac{2(d-2)g}{n_0\paren{d - \frac{2}{r-1},g}},\]
% where $r = \floor{\nicefrac{g}{2}}$, then 
% \[
% {\cegK(G)=(d-2)g.}
% \]
% %the \ceg\ is $(d-2)g.$
% }}
% \end{corollary}

Given the relatively weak spectral condition required in Theorem \ref{thm:main} and the extensive use of the Irregular Moore Bound in the proof, one might naturally wonder whether any spectral condition is required at all.  Here we briefly provide a family of $d$-regular examples (for $d \geq 5$) showing that, for girth 4 at least, the spectral condition is required.  

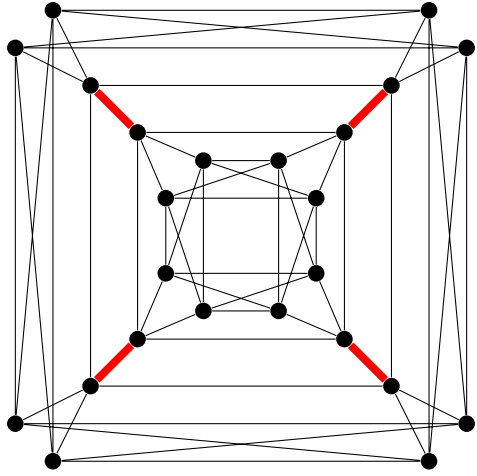
\begin{figure}[t!]
    \centering
    \begin{tikzpicture}
    \def\step{.5}
    \node[vtx] (d1) at (-2*\step,\step) {};
    \node[vtx] (d2) at (\step,2*\step) {};
    \node[vtx] (d3) at (2*\step,-\step) {};
    \node[vtx] (d4) at (-\step,-2*\step) {};
    \node[vtx] (d5) at (-\step,2*\step) {};
    \node[vtx] (d6) at (2*\step,\step) {};
    \node[vtx] (d7) at (\step,-2*\step) {};
    \node[vtx] (d8) at (-2*\step,-\step) {};
    \draw (d1) -- (d2) -- (d3) -- (d4) -- (d5) -- (d6) -- (d7) -- (d8) -- (d1) -- (d4) -- (d7) -- (d2) -- (d5) -- (d8) -- (d3) -- (d6) -- (d1);
    \node[vtx] (c1) at (-2.75*\step,2.75*\step) {};
    \node[vtx] (c2) at (2.75*\step,2.75*\step) {};
    \node[vtx] (c3) at (2.75*\step,-2.75*\step) {};
    \node[vtx] (c4) at (-2.75*\step,-2.75*\step) {};
    \draw (d1) -- (c1) -- (d5);
    \draw (d2) -- (c2) -- (d6);
    \draw (d3) -- (c3) -- (d7);
    \draw (d4) -- (c4) -- (d8);
    \draw (c1) -- (c2) -- (c3) -- (c4) -- (c1);
    \node[vtx] (b1) at (-4*\step,4*\step) {};
    \node[vtx] (b2) at (4*\step,4*\step) {};
    \node[vtx] (b3) at (4*\step,-4*\step) {};
    \node[vtx] (b4) at (-4*\step,-4*\step) {};
    \draw (b1) -- (b2) -- (b3) -- (b4) -- (b1);
    \draw[red,line width = 3] (b1) -- (c1);
    \draw[red,line width = 3] (b2) -- (c2);
    \draw[red,line width = 3] (b3) -- (c3);
    \draw[red,line width = 3] (b4) -- (c4);
    \node[vtx] (a1) at (-6*\step,5*\step) {};
    \node[vtx] (a2) at (5*\step,6*\step) {};
    \node[vtx] (a3) at (6*\step,-5*\step) {};
    \node[vtx] (a4) at (-5*\step,-6*\step) {};
    \node[vtx] (a5) at (-5*\step,6*\step) {};
    \node[vtx] (a6) at (6*\step,5*\step) {};
    \node[vtx] (a7) at (5*\step,-6*\step) {};
    \node[vtx] (a8) at (-6*\step,-5*\step) {};
    \draw (a1) -- (a2) -- (a3) -- (a4) -- (a5) -- (a6) -- (a7) -- (a8) -- (a1) -- (a4) -- (a7) -- (a2) -- (a5) -- (a8) -- (a3) -- (a6) -- (a1);
    \draw (a1) -- (b1) -- (a5);
    \draw (a2) -- (b2) -- (a6);
    \draw (a3) -- (b3) -- (a7);
    \draw (a4) -- (b4) -- (a8);
    \end{tikzpicture}
    \caption{The 5-regular\color{black}\ graph from Example \ref{ex:8edgecut}.  The edge cut of size 4 is bolded and highlighted in red.}
    \label{fig:example}
\end{figure}

\begin{example}\label{ex:8edgecut}

We first note that for a $d$-regular graph with girth $g =4$, \[n_0\paren{d-\frac{2}{r-1},g} = n_0(d-2,4) = 2\paren{1+(d-3)} = 2d-4,\] and thus if $\lambda_2 \leq d - \frac{4d-8}{d-2} = d-4$ then the cyclic edge connectivity is $4d-8$.   In contrast to this, our family of examples has second largest eigenvalue at least $d - \frac{5-\sqrt{17}}{2} > d-4$ and has cyclic edge connectivity {$2d-6$}. 

We begin our construction with the odd case.  In this case the graph can be partitioned into 4 sets, 
\begin{align*}
    A &= \set{a_1,\ldots, a_{4d-12}} \\
    B &= \set{b_1,\ldots,b_{2d-6}} \\
    C &= \set{c_1,\ldots,c_{2d-6}} \\
    D & =\set{d_1,\ldots,d_{4d-12}}.
\end{align*}
Two vertices $a_i, a_j \in A$ are adjacent if $i$ and $j$ differ (cyclically) by one of $1,3,\ldots,d-2$.  A similar relation holds for pairs of vertices in $D$.  Two vertices $b_i$ and $b_j$ are adjacent if $i$ and $j$ differ (cyclically) by one of $1,3,\ldots, d-4$.  All remaining edges are between $A$ and $B$, $B$ and $C$, or $C$ and $D$.  Specifically, $b_i$ is adjacent to $a_i$, $a_{i + d-3}$, and $c_i$.  A analogous relationship holds for vertices in $C$.  It easy to verify that this graph has degree $d$ and girth $4$.  Furthermore, the edge cut between $B$ and $C$ has size $d-1$ with both components containing many cycles.  The $d=5$ case is depicted in Figure \ref{fig:example}.

In the case that $d$ is even, we first generate the graph for degree $d+1$ and then remove a perfect matching contained in each of the sets $A$, $B$, $C$, and $D$.  The existence of such a perfect matching is easily observed by noting that each set contains a even-length cycle through all the vertices. 

Finally we turn to the second largest eigenvalue of each of these graphs.  We note that in both the odd and even cases the sets $A$, $B$, $C$, and $D$ form an equitable partition of the graph with associated equitable partition matrix \[\mathcal{E} = \left[ \begin{matrix} d-1 & 1 & 0 & 0 \\ 2 & d-3 & 1 & 0 \\ 0 & 1 & d-3 & 2 \\ 0 & 0 & 1 & d-1 \end{matrix}\right].\]
The spectrum of $\mathcal{E}$ is $\set{d, d - \frac{5 \pm \sqrt{17}}{2}, d-3}$. Since this is an equitable partition of the graph, the spectrum of the graph includes $d - \frac{5-\sqrt{17}}{2}$.  Furthermore, as the eigenvalues not associated with $\mathcal{E}$ are orthogonal to the indicator vector for each part in the partition (see for instance Section 9.3 of~\cite{Godsil:AGT}), the largest eigenvalue not induced by $\mathcal{E}$ is bounded above by the maximum eigenvalue of 
\[ \left[\begin{matrix} \lambda_2(A) & 1 & 0 & 0 \\ 2 & \lambda_2(B) &  1 & 0 \\ 0 & 1 & \lambda_2(C) & 2 \\ 0 & 0 & 1 & \lambda_2(D) \end{matrix} \right], \]
where $A$, $B$, $C$, and $D$ are the subgraphs induced by the respective sets.  

In the case that $d$ is odd, $A$, $B$, $C$, and $D$ are all Abelian Cayley graphs over a cyclic group and so it straightforward to derive that
\begin{align*}
    \lambda_2(A) = \lambda_2(D) &= \sum_{j=1}^{\frac{d-1}{2}} 2\cos\paren{\frac{(2j-1)\pi}{2d-6}} = \frac{\sin\paren{\frac{(d-1)\pi}{2d-6}}}{\sin\paren{\frac{\pi}{2d-6}}} < d-1\\
    \lambda_2(B) = \lambda_2(C) &= \sum_{j=1}^{\frac{d-3}{2}} 2\cos\paren{\frac{(2j-1)\pi}{d-3}} = 0.
\end{align*}
The largest eigenvalue of 
\[ \left[ \begin{matrix} d -1 & 1 & 0 & 0 \\ 2  & 0 & 1 & 0 \\ 0 & 1 & 0 & 2 \\ 0 & 0 & 1 & d-1\end{matrix}\right] \] 
is $\frac{\sqrt{d^2 -4d +12} + d}{2}$ which is equal to $d - \frac{5-\sqrt{17}}{2}$ when $d =5$ and strictly smaller for larger values of $d$.  

As the example for even $d$ results from removing a repeated matching, the Courant-Weyl inequalities gives that $\lambda_2(B), \lambda_2(C) \leq 1$ while maintaining that $\lambda_2(A),\lambda_2(D) < d-2$. Thus, a similar argument as above yields that for all even $d \geq 4$, the second eigenvalue of every graph in the family is $d - \frac{5-\sqrt{17}}{2}.$
 
\end{example}

We note that the while the above example does not show that Theorem \ref{thm:main} is tight, it does however show that (at least  for girth 4) the bound is of the correct order. \\

\noindent {\bf Acknowledgements.} The authors would like to thank Carlos Ortiz-Marrero for helpful discussions, and anonymous referees for thoughtful comments which improved the manuscript. \\

{ \noindent {\bf Declarations.} This work was supported by the High Performance Data Analytics (HPDA) program at Pacific Northwest National Laboratory.  Pacific Northwest National Laboratory is operated by Battelle Memorial Institute under Contract DE-ACO6-76RL01830. \textit{PNNL Information Release:} PNNL-SA-151831 .  The authors declare they have no competing interests.}

\bibliographystyle{siam}
\bibliography{scRefs}

\begin{thebibliography}{10}

\bibitem{Abiad2018}
{\sc A.~Abiad, B.~Brimkov, X.~Martinez-Rivera, S.~O, and J.~Zhang}, {\em
  Spectral bounds for the connectivity of regular graphs with given order},
  Electronic Journal of Linear Algebra, 34 (2018), pp.~428--443.

\bibitem{Alon2002}
{\sc N.~Alon, S.~Hoory, and N.~Linial}, {\em The {M}oore bound for irregular
  graphs}, Graphs and Combinatorics, 18 (2002), pp.~53--57.

\bibitem{Biggs:AGT}
{\sc N.~Biggs}, {\em Algebraic graph theory}, Cambridge Mathematical Library,
  Cambridge University Press, Cambridge, second~ed., 1993.

\bibitem{boesch1976large}
{\sc F.~T. Boesch}, {\em Large-scale networks, theory and design}, IEEE press,
  1976.

\bibitem{bollobas2004extremal}
{\sc B.~Bollob{\'a}s}, {\em Extremal graph theory}, Dover Books on Mathematics,
  2004.

\bibitem{Bondy2010}
{\sc A.~Bondy}, {\em Graph Theory: (Graduate Texts in Mathematics)}, Springer,
  2010.

\bibitem{cioabua2010eigenvalues}
{\sc S.~M. Cioab{\u{a}}}, {\em Eigenvalues and edge-connectivity of regular
  graphs}, Linear algebra and its applications, 432 (2010), pp.~458--470.

\bibitem{cioabua2012conjecture}
{\sc S.~M. Cioab{\u{a}}, K.~Kim, and J.~H. Koolen}, {\em On a conjecture of
  brouwer involving the connectivity of strongly regular graphs}, Journal of
  Combinatorial Theory, Series A, 119 (2012), pp.~904--922.

\bibitem{cioabua2014disconnecting}
{\sc S.~M. Cioab{\u{a}}, J.~Koolen, and W.~Li}, {\em Disconnecting strongly
  regular graphs}, European Journal of Combinatorics, 38 (2014), pp.~1--11.

\bibitem{dirac63}
{\sc G.~A. Dirac}, {\em Some results concerning the structure of graphs},
  Canad. Math. Bull., 6 (1963), pp.~183--210.

\bibitem{fabrega1994extraconnectivity}
{\sc J.~F{\`a}brega and M.~A. Fiol}, {\em Extraconnectivity of graphs with
  large girth}, Discrete Mathematics, 127 (1994), pp.~163--170.

\bibitem{fan1992integer}
{\sc G.~Fan}, {\em Integer flows and cycle covers}, Journal of Combinatorial
  Theory, Series B, 54 (1992), pp.~113--122.

\bibitem{fiedler1973algebraic}
{\sc M.~Fiedler}, {\em Algebraic connectivity of graphs}, Czechoslovak
  mathematical journal, 23 (1973), pp.~298--305.

\bibitem{Godsil:AGT}
{\sc C.~Godsil and G.~Royle}, {\em Algebraic graph theory}, vol.~207 of
  Graduate Texts in Mathematics, Springer-Verlag, New York, 2001.

\bibitem{Harary1983}
{\sc F.~Harary}, {\em Conditional connectivity}, Networks, 13 (1983),
  pp.~347--357.

\bibitem{latifi1994conditional}
{\sc S.~Latifi, M.~Hegde, and M.~Naraghi-Pour}, {\em Conditional connectivity
  measures for large multiprocessor systems}, IEEE Transactions on Computers,
  43 (1994), pp.~218--222.

\bibitem{liang2020distributed}
{\sc J.~Liang, M.~Du, R.~Nie, Z.~Liang, and Z.~Li}, {\em Distributed algorithms
  for cyclic edge connectivity and cyclic vertex connectivity of cubic graphs},
  in Proceedings of the 2020 4th International Conference on Digital Signal
  Processing, 2020, pp.~279--283.

\bibitem{Lipton1979}
{\sc R.~J. Lipton and R.~E. Tarjan}, {\em A separator theorem for planar
  graphs}, {SIAM} Journal on Applied Mathematics, 36 (1979), pp.~177--189.

\bibitem{liu2010existence}
{\sc Q.~Liu and Z.~Zhang}, {\em The existence and upper bound for two types of
  restricted connectivity}, Discrete Applied Mathematics, 158 (2010),
  pp.~516--521.

\bibitem{lou1993lower}
{\sc D.~Lou and D.~A. Holton}, {\em Lower bound of cyclic edge connectivity for
  n-extendability of regular graphs}, Discrete mathematics, 112 (1993),
  pp.~139--150.

\bibitem{lou2005efficient}
{\sc D.~Lou and W.~Wang}, {\em An efficient algorithm for cyclic edge
  connectivity of regular graphs}, Ars Combinatoria, 77 (2005), pp.~311--318.

\bibitem{lou2008characterization}
{\sc D.~Lou and W.~Wang}, {\em Characterization of graphs with infinite cyclic
  edge connectivity}, Discrete mathematics, 308 (2008), pp.~2094--2103.

\bibitem{lovasz65}
{\sc L.~Lov\'{a}sz}, {\em On graphs not containing independent circuits}, Mat.
  Lapok, 16 (1965), pp.~289--299.

\bibitem{Mohar:LapEig}
{\sc B.~Mohar}, {\em Some applications of laplace eigenvalues of graphs}, in
  Graph symmetry, Springer, 1997, pp.~225--275.

\bibitem{peroche1983several}
{\sc B.~Peroche}, {\em On several sorts of connectivity}, Discrete mathematics,
  46 (1983), pp.~267--277.

\bibitem{plummer1980n}
{\sc M.~D. Plummer}, {\em On n-extendable graphs}, Discrete Mathematics, 31
  (1980), pp.~201--210.

\bibitem{robbins1939theorem}
{\sc H.~E. Robbins}, {\em A theorem on graphs, with an application to a problem
  of traffic control}, The American Mathematical Monthly, 46 (1939),
  pp.~281--283.

\bibitem{tait1880remarks}
{\sc P.~G. Tait}, {\em Remarks on the colouring of maps}, in Proc. Roy. Soc.
  Edinburgh, vol.~10, 1880, pp.~501--503.

\bibitem{wang2009cyclic}
{\sc B.~Wang and Z.~Zhang}, {\em On cyclic edge-connectivity of transitive
  graphs}, Discrete mathematics, 309 (2009), pp.~4555--4563.

\bibitem{zhang2019cyclic}
{\sc W.~Zhang}, {\em The cyclic edge-connectivity of strongly regular graphs},
  Graphs and Combinatorics, 35 (2019), pp.~779--785.

\end{thebibliography}
\end{document}